\documentclass[a4paper,11pt]{article}
\usepackage{bbm}
\usepackage[english]{babel}
\usepackage[style=numeric, backend=bibtex, maxbibnames=99]{biblatex}
\usepackage{amsthm,amsfonts}
\usepackage[latin1]{inputenc}
\usepackage[T1]{fontenc}
\usepackage{amsthm}
\usepackage{makeidx}
\usepackage{latexsym,amsmath,amsfonts,amscd,amssymb}
\usepackage{graphicx}
\usepackage{xypic}
\usepackage{array,booktabs,arydshln}
\usepackage{setspace}
\usepackage{fancyhdr}
\usepackage{fancybox}
\usepackage{multicol}
\usepackage[all]{xy}
\usepackage{pdfpages}
\usepackage{amsmath}
\usepackage{youngtab}
\usepackage{systeme}         
\usepackage{array}       
\usepackage{tikz-cd}     
\usepackage{dsfont}
\usepackage{young}
\usepackage{xpatch}
\usepackage{lipsum}
\usepackage{scrextend}
\usepackage{indentfirst}
\usepackage{dsfont}
\usepackage[hyperfootnotes=false]{hyperref}

\newcommand\blfootnote[1]{%
	\begingroup
	\renewcommand\thefootnote{}\footnote{#1}%
	\addtocounter{footnote}{-1}%
	\endgroup
}
\makeatletter
\AtBeginDocument{\xpatchcmd{\@thm}{\thm@headpunct{.}}{\thm@headpunct{}}{}{}}
\makeatother

\newtheorem*{teorema}{Theorem}
\newtheorem{Teorema}{Theorem}[section]

\newtheorem{Example}[Teorema]{Example}

\newtheorem{cor}[Teorema]{Corollary}
\newtheorem{lema}[Teorema]{Lemma}

\theoremstyle{definition}
\newtheorem{definition}[Teorema]{Definition}

\newcommand{\CC}{{\mathbb C}}

\newcommand{\PP}{{\mathbb P}}

\newcommand{\OO}{{\mathcal O}}
\newcommand{\Sym}{{Sym}}

\newcommand{\Hom}{\operatorname{Hom}}
\newcommand{\End}{\operatorname{End}}

\newcommand{\im}{\operatorname{Im}}

\newcommand{\diag}{\operatorname{diag}}

\makeindex

\title{On tensors that are determined by their singular tuples}
\author{Ettore Turatti}
\date{}
\begin{document}

 	\maketitle
 	\begin{abstract}
 		In this paper we study the locus of singular tuples of a complex valued multisymmetric tensor. The main problem that we focus on is: given the set of singular tuples of some general tensor, which are all the tensors that admit those same singular tuples. Assume that the triangular inequality holds, that is exactly the condition such that the dual variety to the Segre-Veronese variety is an hypersurface, or equivalently, the hyperdeterminant exists. We show in such case that, when at least one component has degree odd, this tensor is projectively unique. On the other hand, if all the degrees are even, the fiber is an $1$-dimensional space.
 	\end{abstract}

 	\section{Introduction}

 	  Let $V_1,\dots,V_k$ be vector spaces over $\mathbb C$ of dimension $m_1+1,\dots,m_k+1$, and let $q_l$ be a quadratic form on $V_l$ that, after a choice of basis for $V_l$, is given in coordinates by $q_l=x_0^2+\dots+x_{m_l}^2$ that defines the distance function in $V_l$. Given a tensor $t\in\Sym^{d_1}V_1\otimes\dots\otimes Sym^{d_k}V_k$, we say that a rank-one tensor $v=v_1\otimes\dots\otimes v_k$ is a singular tuple of $t$ if we have that for each flattening $t:\Sym^{d_1}V_1\otimes\dots\otimes\Sym^{d_l-1}V_l\otimes\dots\otimes \Sym^{d_k}V_k\rightarrow V_l$,  $t(v_1^{d_1}\otimes \dots\otimes v_l^{d_l-1}\otimes \dots \otimes v_k^{d_k})= \lambda v_l$. In the particular case that $k=1$ we call $v$ an eigentensor of $t$, and it satisfies $t(v^{d-1})=\lambda v$. The notion of eigentensor has been introduced in $2005$ by Qi for symmetric tensors \cite{Chang}, \cite{Qi} and it was generalized by Lim \cite{Lim} for tensors as singular tuples in the same year. Moreover, singular tuples extend the notion of eigenvector of a matrix to a tensor of any order.
 	  The first interesting result on singular tuples is proven in \cite{Lim}, the singular tuples of a given tensor $t$ are the critical points of the distance function between $t$ and the Segre-Veronese variety of rank-$1$ tensors. This opens a new perspective to tensor optimization, and a possible direction to obtain a general notion of the Eckart-Young Theorem. Indeed, this theme has been studied in many works, we suggest \cite{ABO2016121}, \cite{Boralevi}, \cite{CARTWRIGHT2013942}, \cite{Horobet}, \cite{Draisma}, \cite{Oeding}, \cite{Paoletti}, \cite{Sodomaco}, \cite{sodomaco2019product}, \cite{Vannieuwenhoven} for a clearer picture of the topic.\blfootnote{Università di Firenze - DIMAI, Viale Morgagni, 67/A, 50134 Firenze, Italy - ettore.teixeiraturatti@unifi.it}\blfootnote{$2020$ Mathematics Subject Classification. $14$N$07$; $14$M$17$; $15$A$18$; $15$A$69$; $58$K$05$.}\blfootnote{\emph{Key words}: Tensor; Eigentensor; Singular tuple; Critical points. }
 	  
 	  A natural question to be posed is, given the singular tuples of a tensor $t$, which are all the tensors that admit such configuration of singular tuples? The answer for the matrix case is described by the {singular value decomposition}, where changing the singular values in the decomposition gives all the matrices with same singular tuples, this is described in Example \ref{Matrix}. The problem for symmetric tensors was first studied on \cite{Abo} for the particular case of $\Sym^d\CC^3$ and $d$ odd, later in \cite{Beorchia} the result obtained has been revisited for general $d$, and Beorchia, Galuppi and Venturello stated the theorem that we show next. We denote by $ed_X$ the $ED$-degree of the Veronese variety $X$. Let $\Sigma_n$ be the $n$-symmetric group of permutations, we denote by $\PP V^{(n)}=(\PP V)^{n}/\Sigma_n$ the symmetric cartesian product such that the points are unordered, let $Eig(f)\subset \PP V^{(ed_X)}$ be the set of eigentensors of a general polynomial $f$ and $q=x_0^2+\dots+x_m^2$ the distance function in a basis of the vector space $V$.
 	  
 	  \begin{teorema}[Beorchia-Galuppi-Venturello (Theorem A, \cite{Beorchia})]
 	  	Let $V$ be a $3$-dimensional vector space, let $$\tau:\PP\big(\Sym^dV\big)\dashrightarrow\PP V^{(ed_X)},\ f\mapsto Eig(f)$$
 	  	be the map that associates a general polynomial $f$ to the set of its eigentensors. Then, if $f\in \PP(\Sym^dV)$ is a general polynomial, we obtain that $$
 	  	\tau^{-1}(\tau(f))=\setlength\arraycolsep{0pt}
 	  	\renewcommand\arraystretch{1.25}
 	  	\left\{
 	  	\begin{array}{llll}
 	  	[f], \text{ if $d$ is odd;} \\ 
 	  	\text{}\{[f + cq^{\frac{d}{2}}]|c\in\CC\}, \text{ if $d$ is even.}
 	  	\end{array}
 	  	\right.
 	  	$$
 	  	
 	  	Moreover, the image of $\tau$ has dimension $$
 	  		\dim(\im(\tau))=\setlength\arraycolsep{0pt}
 	  	\renewcommand\arraystretch{1.25}
 	  	\left\{
 	  	\begin{array}{llll}
 	  	\binom{d+2}{2}-1, \text{ if $d$ is odd;} \\ 
 	  	\text{}\binom{d+2}{2}-2, \text{ if $d$ is even.}
 	  	\end{array}
 	  	\right.
 	  	$$
 	  \end{teorema}
 	  
 	     	The approach utilised in the proof of [Theorem A, \cite{Beorchia}] was difficult to generalize to polynomials in more than three variables, so we opted for a different technique. We notice that we can decompose the map $\tau=\psi\circ\varphi$ into two parts, the first map that we use is the projectivization of the linear map $\varphi:\Sym^dV\rightarrow H^0(Q(d-1))$, where $Q$ is the quotient bundle and a global section $s_f$ associated to a polynomial $f$ is defined by $\varphi(f)=s_f=\begin{bmatrix}
 	     	\nabla f(x)\\
 	     	x
 	     	\end{bmatrix} $; we allow an abuse of notation for simplicity and we denote the projectivization $\PP(\varphi)$ also by $\varphi$. The second map $\psi: \PP(H^0(Q(d-1)))\dashrightarrow \PP V^{(ed_X)}$ takes the zero locus of the section $s_f$ that is exactly the zero dimensional eigenscheme $Eig(f)$. With this approach we were able to generalize [Theorem A, \cite{Beorchia}] to polynomials in any number of variables.
 	     	
 	     	 	\begin{Teorema}\label{fsym}
 	     		Let $V$ be a vector space of dimension $m+1$. Let $d\geq 3$ be an integer, and $f\in \PP(\Sym^dV)$ be a general polynomial. Let $$\tau:\PP\big(\Sym^d V\big)\dashrightarrow \PP V^{(ed_X)}, \ f\mapsto Eig(f)$$
 	     		be the map that associates $f$ to its eigentensors locus $Eig(f)$. Then $$
 	     		\tau^{-1}(\tau(f))=\setlength\arraycolsep{0pt}
 	     		\renewcommand\arraystretch{1.25}
 	     		\left\{
 	     		\begin{array}{llll}
 	     		[f], \text{ if $d$ is odd;} \\ 
 	     		\text{}\{[f + cq^{\frac{d}{2}}]|c\in\CC\}, \text{ if $d$ is even}.
 	     		\end{array}
 	     		\right.
 	     		$$
 	     		
 	     		Moreover, the image of the map $\tau$ has dimension $$
 	     		\dim(\im(\tau))=\setlength\arraycolsep{0pt}
 	     		\renewcommand\arraystretch{1.25}
 	     		\left\{
 	     		\begin{array}{llll}
 	     		\binom{d+m}{d}-1, \text{ if $d$ is odd;}\\ 
 	     		\binom{d+m}{d}-2, \text{ if $d$ is even}.
 	     		\end{array}
 	     		\right.
 	     		$$

 	     	\end{Teorema}
      	
      	The next natural step is to understand what happens in the case of multisymmetric tensors. Although several new technical lemmas are required, the approach to generalize this result is similar to the symmetric tensor case. Let $X\subset \PP\big(\Sym^{d_1}V_1\otimes\dots\otimes \Sym^{d_k}V_k\big)$ be the Segre-Veronese variety, $\pi_i:X\rightarrow \PP V_i$ be the projection on the $i$-th coordinate and $\tau:\PP\big(\Sym^{d_1}V_1\otimes\dots\otimes \Sym^{d_k}V_k\big)\dashrightarrow (\PP V_1\times\dots\times \PP V_k)^{(ed_X)}$ be the map that associates a multisymmetric tensor $T$ to the its eigenscheme $Eig(T)$, i.e. the locus of its singular tuples. We construct the bundle $\mathcal E=\bigoplus \pi_i^\ast Q_i(d_1,\dots,d_{i-1},d_i-1,d_{i+1},\dots,d_k)$ and use the fact that the zero locus of a global section $s_T\in H^0(\mathcal E)$ associated to a multisymmetric tensor $T$ is the singular tuples locus of $T$, as described in more details in both \cite{Draisma} and \cite{Friedland}, to split the map $\tau=\psi\circ\varphi$ as in the symmetric tensor case, where here we consider the projectivization of $\varphi:\Sym^{d_1}V_1\otimes \dots\otimes \Sym^{d_k}V_k\rightarrow H^0(\mathcal E)$ and $\psi:\PP(H^0(\mathcal E))\rightarrow (\PP V_1\times \dots\times \PP V_k)^{(ed_X)}$. 
      	
      	\begin{Teorema}\label{fmulti}
      		Let $V_1,\dots,V_k$ be vector spaces of dimension $m_1+1,\dots,m_k+1$. Let $d_1,\dots,d_k$ be positive integers, and $T\in \PP\big(\Sym^{d_1}V_1\otimes \dots\otimes \Sym^{d_k}V_k\big)$ be a general tensor. Let $$\tau:\PP\big(\Sym^{d_1}V_1\otimes \dots\otimes \Sym^{d_k}V_k\big)\dashrightarrow (\PP V_1\times\dots\times \PP V_k)^{(ed_X)}, T\mapsto Eig(T),$$
      		be the map that associates a tensor $T$ to its singular tuples locus $Eig(T)$. If $k\geq 3$ and suppose that $m_l\leq \sum_{j\neq l}m_j$ whenever $d_l=1$, and for $k=2$ we include the hypothesis that $(d_1,d_2)\neq (1,1)$, then $$
      		\tau^{-1}(\tau (T))=\setlength\arraycolsep{0pt}
      		\renewcommand\arraystretch{1.25}
      		\left\{
      		\begin{array}{llll}
      		[T], \text{ if $d_i$ is odd for some }i; \\ 
      		\{[T + cq_1^{\frac{d_1}{2}}\otimes\dots\otimes q_k^{\frac{d_k}{2}}]|c\in \CC\}, \text{ if $d_l$ is even for all } l.
      		\end{array}
      		\right.
      		$$ 
      		
      		Moreover, the image of the map $\tau$ has dimension $$
      	\dim(\im(\tau))=\setlength\arraycolsep{0pt}
      	\renewcommand\arraystretch{1.25}
      	\left\{
      	\begin{array}{llll}
      	\prod_{l=1}^k\binom{d_l+m_l}{d}-1, \text{ if $d_i$ is odd for some }i;\\ 
      	\prod_{l=1}^k\binom{d_l+m_l}{d}-2, \text{ if $d_l$ is even for all } l.
      	\end{array}
      	\right.
      	$$
      		 
      	\end{Teorema}	
      
      The hypothesis of the triangular inequality $m_l\leq\sum_{i\neq l}m_i$, that also appears in the description of the codimension of the critical space in \cite{Draisma}, could seem unnatural at first glance, but this condition can be understood in terms of the dual variety of the Segre-Veronese variety, as described in the next theorem.
      \begin{teorema}[Gelfand-Kapranov-Zelevinsky, Corollary 5.11, \cite{Gelfand}]
      	Suppose $X_l$ for $l=1,\dots,k$ is the projective space $\PP^{m_l}$ in the Veronese embedding into $\PP({\Sym^{d_l}V_l})$. Then the dual variety $(X_1\times\dots\times X_k)^\vee$ is a hypersurface if and only if $m_l\leq\sum_{i\neq l}m_i$ hold for all $l$ such that $d_l=1$. 
      \end{teorema}
  
  	The case when we have that $d_l=1$ and the equality on the triangular inequality $m_l\leq\sum_{i\neq l}m_i$ holds is called boundary format case.

  	The article is divided into three parts. Section two is a preliminaries section, where we introduce with more details the singular tuples and we give the cohomological tools that are necessary for our results. In the third section we work on the symmetric tensor case and prove Theorem \ref{fsym}. In the final section we analyse the multisymmetric tensor case and prove Theorem \ref{fmulti}.
  \section*{Acknowledgement}   
   
  The author would like to thank Giorgio Ottaviani for proposing this work and for the valuables discussions, suggestions and encouragement.
  
  This work has been supported by European Union's Horizon 2020 research and innovation programme under the Marie Sk\l odowska-Curie Actions, grant agreement 813211 (POEMA).

 	
 	\section{Preliminaries}
 	
 	\subsection{Eigentensors and singular tuples}
 	We suggest both \cite{landsbergtensors} and \cite{qi2017tensor} as references for a deeper understanding of the notions presented in this section.
 	
 	\begin{definition}
	 		Suppose $f\in\Sym^dV$ is a symmetric tensor, i.e., a homogenous polynomial and $q$ a quadratic form on $V$, we will assume in this paper that, after some change basis of $V$, it is given in coordinates by $q=x_0^2+\dots+x_m^2$, the eigenvectors of $f$ are defined in \cite{Qi} as the vectors $x\in V$ such that $$f(x^{d-1})=\lambda x,$$ for $\lambda\in \CC$. In other words, we can define the eigentensors of $f$ as the fixed points of $\nabla f=\big(\frac{\partial f}{\partial x_0}(x),\dots,\frac{\partial f}{\partial x_m} (x)\big)$, the gradient vector, i.e. the solutions of the equation $$
 		\nabla f(x)=\lambda x.
 		$$
 		In conclusion, the equations of the locus of eigentensors $Eig(f)$ is determined by the $2\times2$-minors of the matrix
 		$$
 		\begin{bmatrix}
 		\frac{\partial f}{\partial x_0}&\dots&\frac{\partial f}{\partial x_m}\\
 		x_0&\dots&x_m
 		\end{bmatrix}.
 		$$ 
 	\end{definition}
 	
 	\begin{definition}
 		Let $T\in \Sym^{d_1}V_1\otimes\dots\otimes \Sym^{d_k}V_k$ be a tensor and $q_i=x_0^2+\dots+x_{m_i}^2$ a quadratic form on $V_i$ for a choice of basis for each $V_i$. We can define the singular tuples of $T$ as the tuple $(v_1,\dots,v_n)$, such that each flattening $$T_i:\Sym^{d_1}V_1\otimes\dots\otimes \Sym^{d_i-1}{V_i}\otimes\dots\otimes \Sym^{d_k}V_k\rightarrow V_i 
 		$$
 		satisfies $$T_i(v_1^{d_1}\otimes\dots\otimes v_i^{d_i-1}\otimes\dots\otimes v_k^{d_k})=\lambda v_i.$$
 		
 		We define the zero dimensional scheme $Eig(T)$ to be the locus of singular tuples of $T$.
 	\end{definition}
 	
 	
 	
 	\begin{Example}\label{Matrix}
 
	The relation between a given set of singular tuples and the matrices that have such singular tuple locus configuration is described by the Singular Value Decomposition, since the singular tuples of a matrix are given by the first columns of the orthogonal matrices on the decomposition. We briefly describe it next.

	Let $A\in \Hom( Y,W)$, where $Y,\  W$ are vector spaces of dimensions $\dim Y=n,\ \dim{W}=m$, we recall that the singular value decomposition tells us that $A=U\diag(\sigma_1,\dots,\sigma_{	\min\{m,n\}})V^t$. If we let $u_i$ and $v_i$ be the columns of $U$ and $V$, as described by Ottaviani and Paoletti in \cite{Paoletti}, we have that for $1\leq i\leq m=\min\{m,n\}$, $Av_i=u_i$ and $A^tu_i=v_i$, in other words, the pairs $(u_i,v_i)$ are the singular pairs of $A$. Let $\tau:\Hom(Y,{W})\rightarrow (Y\times {W})^{(m)},\ A\mapsto Eig(A),$ where $Eig(A)$ is the set consisting of the singular tuples of $A$. Therefore, given a singular tuple locus $Z=\{(u_i,v_i)\}_{i=1}^{m}$, and orthogonal matrices $U,V$ such that the first $m$ columns are $u_i$ and $v_i$ we have that $$
	\tau^{-1}(Z)=\{B\in \Hom(Y,{W})|B=U\diag(\sigma_1,\dots,\sigma_m)V^t, \sigma_i\in\CC \}.
	$$
 	\end{Example}

 	 Let $V_1,\dots,V_k$ be vector spaces of respective dimension $m_1+1,\dots,m_k+1$, and let $T\in \Sym^{d_1}V_1\otimes \dots \otimes\Sym^{d_k}V_k$.
 	 
 	 \begin{definition}\label{E}
 	 	Let $X=\PP V_1\times \dots\times \PP V_k$ be the Segre-Veronese variety of rank $1$ tensors embedded with $\OO(d_1,\dots,d_k)$ in $\PP(\Sym^{d_1}V_1\otimes \dots\otimes \Sym^{d_k}V_k)$. Let $\pi_l:X\rightarrow \PP V_l$ be the projection on the $l$-th component, and let $Q_l$ be the quotient bundle, whose fibers over a point $v_l\in V_l$ are $V_l/\langle v_l\rangle$. Let $\mathcal E_l=\pi_l^\ast Q_l\otimes \OO(d_1,\dots,d_{l}-1,\dots,d_l) $, we can construct the vector bundle $$
 	 \mathcal E=\bigoplus_{l=1}^k \mathcal E_l.
 	 $$
 	 A tensor $T\in\PP(\Sym^{d_1}V_1\otimes \dots\otimes \Sym^{d_k}V_k)$ leads to a global section of $\mathcal E_l$ which over a point $v=(v_1,\dots,v_k)$ is the map sending $v$ to the natural pairing of $f$ with $(v_1^{d_1})\cdots (v_l^{d_l-1})\cdots (v_k^{d_k})$ modulo $\langle v_l\rangle $, that is a vector in $V_l/\langle v_l\rangle$. 
 	 \end{definition}
 	
 	The reason to consider this particular bundle is that in \cite{Draisma} and \cite{Friedland} it is proven that if we consider the section $s_T$ associated to a multisymmetric tensor $T$, then the zero locus $Z(s_T)$ is equal to the locus of singular tuples of $T$, that is $Z(s_T)=Eig(T)$. In particular Friedland and Ottaviani \cite{Friedland} used this fact to compute the number of singular tuples of a general tensor as the top Chern class of the bundle $\mathcal E$ in Theorem \ref{Friedland}.
 	\begin{definition}
 		
 	The $ED$-degree of a subvariety $X\subset \PP(\Sym^{d_1}V_1\otimes\dots\otimes \Sym^{d_k}V_k)$ is defined as the number of critical points of the function $d_T:X\rightarrow \mathbb R$, where $d_T$ is the distance function between a general tensor $T\in \PP(\Sym^{d_1}V_1\otimes\dots\otimes \Sym^{d_k}V_k)$ and $X$. We consider the natural extension of this function to the complex numbers for our purposes. 
 	
 	\end{definition}
 	The $ED$-degree has been studied in \cite{Horobet}, and we suggest it as a reference for a better comprehension. In particular, if we consider the variety $X$ to be the Segre-Veronese variety, we have that the $ED$-degree counts the number of singular tuples of a general tensor. We are going to denote the $ED$-degree of the Segre-Veronese variety by $ed_X$. This particular $ED$-degree has been studied before in \cite{Friedland}, where the next theorem is presented.
 	
 	\begin{Teorema}[\cite{Friedland}, Theorem $15$]\label{Friedland}
 		Let $V_1,\dots,V_k$ be vector spaces of dimension $m_1+1,\dots,m_k+1$. The number of singular tuples of a general tensor $T\in \PP(\Sym^{d_1}V_1\otimes\dots\otimes\Sym^{d_k} V_k)$, is equal to the coefficient of $t_1^{m_1}\cdots t_k^{m_k}$ in the polynomial $$
 		\prod_{l=1}^{k}\frac{\hat{t_l}^{m_l+1}-t_l^{m_l+1}}{\hat{t_l}-t_l}
 		$$
 		where $\hat t_l=(\sum_{i}^kd_it_i) -t_l$.
 		
 	\end{Teorema}  

	\begin{Example}
		Assume that $d_i=m_i=1$, for all $i=1,\dots,k$, in this setting we can compute using the previous formula that the $ED$-degree of the Segre variety $X$ is $k!$.
	\end{Example}
	\begin{Example}
		Furthermore, the number of singular tuples stabilises at the boundary format case, that is, suppose that $d_k=1$ and $N=\sum_{i=1}^{k-1}m_i\leq m_k$, then the number of singular tuples of a general tensor is constant as $m_k$ increases.
	\end{Example}

	Consider the map $$\tau:\PP\big(\Sym^{d_1}V_1\otimes\dots\otimes \Sym^{d_k}V_k\big)\dashrightarrow \big[\PP(V_1)\times\dots\times \PP(V_k)\big]^{(ed_X)}$$
	that for a general tensor $T\in \Sym^{d_1}V_1\otimes\dots\otimes \Sym^{d_k}V_k$ it associates its singular tuples locus $Eig(T)$. Studying this map is difficult in general, but decomposing it through the bundle $\mathcal E$ is advantageous. We decompose $\tau$ in the following manner $$
	\begin{tikzcd}[column sep=scriptsize]\PP\big(\Sym^{d_1}V_1\otimes\dots\otimes\Sym^{d_k}V_k\big)\arrow[dashed,dr,"\varphi"]\arrow[dashed,rr,"\tau"]{}& & \big[\PP(V_1)\times\dots\times \PP(V_k)\big]^{(ed_X)} \\& \PP(H^0(\mathcal E))\arrow[dashed,ru,"\psi"]\end{tikzcd}	
	$$
	In this diagram the map $\varphi$ associates a tensor $T$ to the global section $s_T$ described before in the definition \ref{E}. The map $\psi$ sends a global section $s\in H^0(\mathcal E)$ to its zero locus $Z(s)$, in particular the codomain is well defined for a section $s_T$ when the singular tuples of $T$ consists of exactly $ed_X$ points. 
	\begin{Example}
		Notice that for $k=1$ we obtain the symmetric tensor case, in such case $\mathcal E$ is simply equal to $Q(d-1)$ and the map $\varphi$ can be described as $$
		\varphi:\Sym^d V\rightarrow H^0(Q(d-1)),\ f\mapsto s_f=\begin{bmatrix}
		\nabla f(x)\\
		x
		\end{bmatrix}
		$$
		
		The other interesting case is when $d_l=1$ for all $l=1,\dots,k$. In such case, $X$ is the Segre variety and the map $\varphi$ can be described by means of the flattenings of the tensor $T$, that is $$
		\varphi: V_1\otimes \dots\otimes V_k\rightarrow \bigoplus_{l=1}^k\Hom(V_1^\ast\otimes\dots\otimes\widehat{V_l^\ast}\otimes \dots\otimes V_k^\ast,V_l),\ T\mapsto (T_1,\dots,T_k);
		$$
		In our notation $T_l$ represents the $l$-flattening of T, namely $$
		T_l:V_1\otimes\dots\otimes \hat V_l\otimes\dots\otimes V_k\rightarrow V_l^\ast.
		$$
		
		In the general case $X$ is the Segre-Veronese variety; this is the case of multisymmetric tensors. The map $\varphi$ in this case is a combination of the previous two, that is, in each $l$-th component the maps acts as the contraction in the $l$-th coordinate and the evaluation in the others. 
	\end{Example}

 	\subsection{Cohomological ingredients}
 	
 	We recall the next classical concepts and results that will be utilised in the course of this article, we suggest \cite{Weyman} for more details.
 	
 	\begin{teorema}[K\"unneth's formula]
 		Let $\mathcal B_i$ be vector bundles on $\PP V_i$, $i=1,\dots,k$ and $q$ a non-negative integer, then $$
 		H^q\bigg(\bigotimes_{i=1}^k \pi_i^\ast \mathcal B_i\bigg)\cong \bigoplus_{q_1+\dots+q_k=q}\bigotimes_i H^{q_i}(\mathcal B_i).
 		$$
 		where the sum goes over all tuples of non-negative integers summing $q$.
 		
 	\end{teorema}

Let $G$ be a semisimple simply connected group, let $P\subset G$ be a parabolic subgroup. Let $\Phi^+$ be the set of positive roots of $G$. Let $\delta=\sum \lambda_i$ be the sum of all the fundamental weights and let $\lambda$ be a weight. Let $E_\lambda$ be the homogeneous bundle arising from the irreducible representation of $P$ with highest weight $\lambda$ and $(\ ,\ )$ be the Killing form.
 
 \begin{definition}
 	The weight $\lambda$ is called singular if there exists a root $\alpha \in \Phi^+$ such that $(\lambda,\alpha)=0$. Otherwise, if $(\lambda,\alpha)\neq 0$ for all the roots $\alpha \in \Phi^+$, we say that $\lambda$ is regular of index $p$ if there exists exactly $p$ roots $\alpha_1,\dots,\alpha_p\in\Phi^+$ such that $(\lambda,\alpha)<0$.
 \end{definition}

 \begin{teorema}[Bott]
 		  If $\delta+\lambda$ is singular, then $H^i(G/P, E_\lambda)=0$ for all $i$.
 		 If $\delta+\lambda$ is regular of index $p$, then $H^i(G/P,E_\lambda)=0$ for $i\neq p$.

 \end{teorema}

 	\section{Symmetric Tensors}
 	\begin{lema}\label{symmetric}
 		Let $V$ be a vector space of dimension $m+1$, $q=x_0^2+\dots+x_m^2$ a quadratic form on $V$, and $d$ a positive integer. If $d$ is odd, the map $\varphi:\Sym^d V\rightarrow H^0(Q(d-1))$ is injective. If $d$ is even, $\varphi$ has a $1$-dimensional kernel, namely, $\ker\varphi=\langle q^{d/2} \rangle$.
 	\end{lema}
 \begin{proof}
 		
 	We recall that $\Sym^dV$ splits as $SO(V)$-modules as $$
 	\Sym^dV=H_d\oplus H_{d-2}\oplus \dots\oplus
 	\renewcommand\arraystretch{1.25}
 	\left\{
 	\begin{array}{llll}
 	H_1 \text{ if $d$ is odd} \\
 	H_0 \text{ if $d$ is even},
 	\end{array}
 	\right.
 	$$
 	where $H_{d-2j}=\{fq^{j}|f\text{ is a harmonic polynomial of degree }d-2j\}$ is an irreducible $SO(V)$-module.
 	
 	Therefore we can restrict $\varphi$ to each $H_j$, in such way we have $$\varphi:H_j\rightarrow W_j\subset H^0(Q(d-1)),$$ where $W_j=\im(\varphi)|_{H_j}$. This map is either an isomorphism or zero by Schur's lemma. Let $j$ be such that $d-2j\geq 1$, then we have that for $g=(x_0+ix_1)^{d-2j}q^{j}\in H_{d-2j}$ it is mapped by $\varphi$ to $$s_g=\begin{bmatrix}
 	\frac{\partial g}{\partial x_0}&\dots&\frac{\partial g}{\partial x_m}\\
 	x_0&\dots&x_m
 	\end{bmatrix}$$ that has not rank $1$ everywhere. Indeed $$
 	\frac{\partial g}{\partial x_0} x_1-\frac{\partial g}{\partial x_1} x_0 = \big((d-2j)(x_0+ix_1)^{d-2j-1}\big)\big(x_1-ix_0\big)q^j\not\equiv 0.
 	$$
 	On the other hand, $H_0=\{\lambda q^{\frac{d}{2}}|\lambda\in\CC\}$. In such case we have for an element of $H_0$ that $$
 	\frac{\partial \lambda q^{\frac{d}{2}}}{\partial x_i}x_j-\frac{\partial \lambda q^{\frac{d}{2}}}{\partial x_j}x_i=\lambda(2x_ix_jq^{\frac{d}{2}-1}-2x_ix_jq^{\frac{d}{2}-1})=0,\ \forall\  i,j\in\{0,\dots,m\}.
 	$$
 	We conclude that if $d$ is odd, the map $\varphi$ is an isomorphism in each irreducible representation; if $d$ is even, it is an isomorphism in each of them, with the exception of $H_0$, as we wished.
 \end{proof}
\begin{lema}{\label{h0 sym}}
	Let $Z$ be the zero locus of a section in $Q(d-1)$, and assume that $d\geq 3$. Then the natural map  from the Koszul complex $H^0(\End Q)\rightarrow H^0(\mathcal I_Z\otimes Q(d-1))$ is an isomorphism of $1$-dimensional spaces.
\end{lema}
\begin{proof}
 	Indeed, consider the Koszul complex 
 	$$
 	0\xrightarrow{\varphi_{m} } \bigwedge^{m}Q^\ast(m(1-d))\xrightarrow{\varphi_{m-1}} \dots\xrightarrow{\varphi_2} \bigwedge^2Q^\ast(2(1-d)) \rightarrow Q^\ast(1-d)\rightarrow \mathcal I_Z\rightarrow 0,
 	$$
 	tensoring it by $Q(d-1)$ we obtain the exact sequence
 	$$
 	0\rightarrow \bigwedge^{m}Q^\ast\otimes Q((m-1)(1-d))\rightarrow \dots\rightarrow \bigwedge^2Q^\ast\otimes Q(1-d)\rightarrow \End(Q)\rightarrow \mathcal I_Z\otimes Q(d-1)\rightarrow 0.
 	$$
 	
 		Let $\mathcal F_r$ to be defined as the quotient $\mathcal F_r=\bigwedge^r Q^\ast(r(1-d))/\im \varphi_r$. Thus we obtain short exact sequences $$
 	\begin{matrix}
 	0\rightarrow\mathcal F_2\rightarrow  Q^\ast(1-d)\rightarrow \mathcal I_Z\rightarrow 0\\
 	0\rightarrow\mathcal F_{r+1}\rightarrow\bigwedge^r  Q^\ast(r-rd)\rightarrow \mathcal F_{r}\rightarrow 0,
 	\end{matrix}
 	$$
 	for $r=2,\dots, m$.
 	
 	Tensoring the second short exact sequence by $Q(d-1)$ we obtain the long exact sequence of cohomologies
 	
 	\begin{equation*}
 	\begin{aligned}
 	\dots\rightarrow H^{r-2}(\bigwedge^{r}Q^\ast\otimes Q((r-1)(1-d)))\rightarrow H^{r-2}(\mathcal F_r\otimes Q(d-1))\rightarrow H^{r-1}(\mathcal F_{r+1}\otimes Q(d-1))\rightarrow\\
 	\rightarrow H^{r-1}(\bigwedge Q^\ast\otimes Q((r-1)(1-d)))\rightarrow H^{r-1}(\mathcal F_r\otimes Q(d-1))\rightarrow H^r(\mathcal F_{r+1}\otimes Q(d-1))\rightarrow \dots
 	\end{aligned}
 	\end{equation*}

 		We have that $ \bigwedge^{r}Q^\ast\otimes Q((r-1)(1-d))=\bigwedge^{m-r}Q\otimes Q((r-1)(1-d)-1)$, so if we have that $r\geq 2$, we obtain that $H^{r-2}(\bigwedge^{m-r}Q\otimes Q((r-1)(1-d)-1))=0$. Also, if $d\geq 3$, $H^{r-1}(\bigwedge^{m-r}Q\otimes Q((r-1)(1-d)-1))=0$.
 		
 		This means that $$
 		\begin{matrix}
 		H^{0}(\mathcal F_2\otimes Q(d-1))\cong H^1(\mathcal F_3\otimes Q(d-1))\cong\dots\cong H^{m-1}(\mathcal F_{m+1}\otimes Q(d-1))=0\\
 		H^{1}(\mathcal F_2\otimes Q(d-1))\subset H^2(\mathcal F_3\otimes Q(d-1))\subset\dots\subset H^{m}(\mathcal F_{m+1}\otimes Q(d-1))=0
 		\end{matrix}
 		$$
 		Applying the long exact sequence of cohomologies to $$0\rightarrow \mathcal F_2\otimes Q(d-1)\rightarrow \End(Q)\rightarrow \mathcal I_Z\otimes Q(d-1)\rightarrow0$$
 		gives the desired result.
 	\end{proof}
 	
 	We would like to add a remark that, although already utilised, the vanishing of the cohomology $H^q(\bigwedge^rQ^\ast\otimes Q(t))$ is carefully done in the next section on Lemma \ref{bott}. We decide in favour of postponing those computations because the full usefulness of such cohomologies appears in the multisymmetric case.
 	
 	\begin{cor}\label{cor sym}
 		Let $f, g\in \Sym^d V$ be two general polynomials such that $Z(s_f)=Z(s_g)$, $d\geq 3$. Then $s_f=\alpha s_g$ for some $\alpha\in \CC^\ast$.
 	\end{cor} 
 \begin{proof}
 	The hypothesis that $Z(s_f)=Z(s_g)$ implies that $s_f\in H^0(\mathcal I_{Z(s_g)}\otimes Q(d-1))$. Since this space is one-dimensional we have that $s_f=\alpha s_g$.

 \end{proof}

We conclude this section observing that since $\tau=\psi\circ \varphi$, then the Theorem \ref{fsym} is obtained just as the combinination of the Lemma \ref{symmetric} with the Corollary \ref{cor sym}.

\section{Multisymmetric Tensors}

Now that the pre-image of the map $\tau$ is completely analysed for symmetric tensors, we can go through to the next step, that is, we consider the Segre-Veronese variety $	\Sym^{d_1}V_1\otimes\dots\otimes \Sym^{d_k}V_k$ and we analyse the map $\tau:\PP\big(\Sym^{d_1}V_1\otimes\dots\otimes \Sym^{d_k}V_k\big)\rightarrow\big[\PP(V_1)\times\dots\times \PP(V_k)\big]^{(ed_X)}$ that associates a tensor $T$ to its singular tuples $Eig(T)$. We begin the multisymmetric case with the generalization of the Lemma \ref{symmetric} to Segre-Veronese varieties.   

	\begin{Teorema}\label{segrever}
	Let $V_1,\dots,V_k$ be vector spaces of dimension $m_1+1,\dots,m_k+1$, and we recall that  $q_i=x_0^2+\dots+x_{m_i}^2$ is the quadratic form on $V_i$ that defines the distance function for $i=1\,\dots,k$. We consider the map $$\varphi: \Sym^{d_1}V_1\otimes \dots\otimes \Sym^{d_k}V_k\rightarrow H^0(\mathcal E),$$ where $\mathcal E$ is defined in the Definition \ref{E}. Then $\varphi$ is injective if at least one $d_i$ is odd. In the case that all the $d_i$ are even, we have that the kernel of $\varphi$ is one dimensional and it is given by $$
	\ker \varphi=\langle q_1^{\frac{d_1}{2}}\rangle\otimes \dots\otimes \langle q_k^{\frac{d_k}{2}}\rangle
	$$
\end{Teorema}

\begin{proof}
	Since we have that $$\Sym^{d_l}V_l\cong H_{d_l}\oplus H_{d_l-2}\oplus\dots \oplus\renewcommand\arraystretch{1.25}
	\left\{
	\begin{array}{llll}
	H_1 \text{ if $d_l$ is odd} \\
	H_0 \text{ if $d_l$ is even},
	\end{array}
	\right.$$
	and that each $H_{d_j-2t_j}$ is an irreducible $SO(V_l)$-representation, then also $H_{d_1-2t_1}\otimes\dots \otimes H_{d_k-2t_k}$ is an irreducible $SO(V_1)\times\dots\times SO(V_k)$-representation, we need to show that $\varphi$ is non zero when $d_j-2t_j>0$ for at least one $j$, and that it is zero when we have $d_j-2t_j=0$ for all $j$.
	
	Indeed, in the first case we consider the element $$
	g=g_1\otimes\dots\otimes g_k,\ g_j=(x_0+ix_1)^{k-2j}q^{t_j},
	$$
	then $\varphi(g)=s_g=(s_{g_1}\otimes \mathds{1}) \oplus\dots\oplus (\mathds{1}\otimes s_{g_k})$, where $s_{g_j}\otimes\mathds{1}\in\mathcal E_j$ is non zero as seen before in the symmetric tensor case. Therefore by Schur's lemma we have that in this restriction the map is an isomorphism, thus if $d_j-2t_j>0$ for some $j$, $s_g$ does not belong to the kernel of $\varphi$.
	
	On the other hand, if all $d_j-2t_j=0$, then $g_{j}=c q^{\frac{d_j}{2}}$, where $c\in\CC$, then $s_{g_j}=0$, therefore summing all together we obtain that $s_g=0$, so by Schur's Lemma the restriction of $\varphi$ on this subrepresentations is the zero map, as wished.
\end{proof}

With this result we understand the first map $\varphi$ in the decomposition $\tau=\psi\circ\varphi$. Now we can aim to understand better the map $\psi$, we will show that, under the hypothesis of Theorem \ref{fmulti}, when two section $s,\ t$ have the same image under the map $\psi$, where $s,\ t$ are sections coming from tensors $S,\ T\in\Sym^{d_1}V_1\otimes \dots\otimes \Sym^{d_k}V_k$, then $s=\lambda t$.

The first step to achieve this goal is to prove a similar result to Lemma \ref{h0 sym} for the case of multisymmetric tensors, in order to do that we prove a series of technical lemmas.
\begin{lema}
	Let $\mathcal E^\ast=\bigoplus_{l=1}^k Q_l^\ast(-d_1,\dots,-d_l+1,\dots,-d_k)$, then, for $j=1,\dots, k$, then 
\begin{equation*}
\begin{aligned}
\bigwedge^r\mathcal  E^\ast&\otimes Q_j(d_1,\dots,d_j-1,\dots,d_k)=\\ &=\bigoplus_{r_1+\dots+r_k=r}\bigotimes_{l=1,l\neq j}^k\Omega^{r_l}_{\PP^{m_l}}(2r_l-d_l(r-1))\otimes\bigwedge^{m_j-r_j}Q_j\otimes Q_j(-d_j(r-1)+r_j-2).
\end{aligned}
\end{equation*}

\end{lema}
\begin{proof}
	From the definition of $\mathcal E$ we have that
	$$
	\bigwedge^r \mathcal E^\ast=\bigoplus_{r_1+\dots+r_k=r}\bigotimes_{l=1}^k\big( \bigwedge^{r_l}Q_l^\ast\big)\big(-r_ld_1,\dots,-r_l(d_l-1),\dots,-r_ld_k) \big),
	$$
	by separating the terms we obtain that $$
	\bigwedge^r \mathcal E^\ast=\bigoplus_{r_1+\dots+r_k=r}\bigotimes_{l=1}^k\bigwedge^{r_l}Q^\ast_l(-rd_l+r_l).
	$$
	We now tensor it by $Q_j(d_1,\dots,d_j-1,\dots,d_k)$, so we have that $\bigwedge^r\mathcal  E^\ast\otimes Q_j(d_1,\dots,d_j-1,\dots,d_k)$ is equal to $$
	\bigoplus_{r_1+\dots+r_k=r}\bigotimes_{l=1,l\neq j}^k\bigwedge^{r_l}Q^\ast_l(-rd_l+r_l+d_l)\otimes\bigwedge^{r_j}Q^\ast_j\otimes Q_j(-rd_j+r_j+d_j-1).
	$$
	We now use the facts that $\Omega^{r_l}(r_l)=\bigwedge^{r_l}(\Omega^1(1))$, $\Omega^1(1)=Q^\ast$, and $\bigwedge^{r_j}Q_j^\ast=\bigwedge^{m_j-r_j}Q_j(-1)$, to obtain that $\bigwedge^r \mathcal E^\ast\otimes Q_j(d_1,\dots,d_j-1,\dots,d_k)$ is equal to 
	$$
	\bigoplus_{r_1+\dots+r_k=r}\bigotimes_{l=1,l\neq j}^k\Omega^{r_l}_{\PP^{m_l}}(2r_l-d_l(r-1))\otimes\bigwedge^{m_j-r_j}Q_j\otimes Q_j(-d_j(r-1)+r_j-2).
	$$
\end{proof}

\begin{lema}[Bott's Theorem]\label{bott}
	The cohomology $H^q(\bigwedge^{m_j-r_j}Q_j\otimes Q_j(t))$ is non vanishing for the following cases\begin{equation}
	H^q\bigg(\bigwedge^{m_j-r_j}Q_j\otimes Q_j(t)\bigg)\neq 0, \text{ if } \setlength\arraycolsep{0pt}
	\renewcommand\arraystretch{1.25}
	\left\{
	\begin{array}{llll}
	q=0,\ t\geq 0,\\
	q=r_j-1,\ t=-r_j,\ 1\leq r_j\leq m_j, \\
	q=r_j,\ t=-r_j-1,\ 0\leq r_j\leq m_j-1,\\
	q=m_j-1,\ t=-m_j-1,\ 0\leq r_j\leq m_j-1,\\
	q=m_j,\ t\leq-m_j-2.
	\end{array}
	\right.
	\end{equation}

\end{lema}
\begin{proof}
	The associated weight will be calculated in three cases depending on the $r_j$; the cases are $r_j=0$, $1\leq r_j\leq m_j-1$ and $r_j=m_j$. 
	
	For the case $1\leq r_j\leq m_j-1$, we have that $\bigwedge^{m_j-r_j}Q_j\otimes Q_j(t)$ is not irreducible, therefore we have that the associated weight $\lambda$ is given by two parts
	$$
	\lambda=\lambda_{(1)}\oplus\lambda_{(2)}.
	$$
	where 
	$\lambda_{(1)}=\lambda_{r_j+1}+\lambda_{m_j}+t\lambda_1$ and $\lambda_{(2)}=\lambda_{r_j}+t\lambda_1$.
	
	For $\lambda_{(1)}$ we have that 
	$$
	(\lambda_{(1)}+\delta,\alpha_1+\dots+\alpha_s)=\setlength\arraycolsep{0pt}
	\renewcommand\arraystretch{1.25}
	\left\{
	\begin{array}{llll}
	s+t \text{ if $s\leq r_j$,} \\
	s+t+1 \text{ if $r_j+1\leq s\leq m_j-1$,}\\
	s+t+2 \text{ if $s= m_j$.}
	\end{array}
	\right.
	$$
	This implies the following cases:\begin{enumerate}
		\item $t\geq0$, then index $0$.
		\item $-1\geq t\geq-r_j$, then it is singular ($s=-t$ gives the vanishing).
		\item If $t=-r_{j}-1$, then index $r_j$.
		\item If $-r_j-2\geq t\geq -m_j$, then it is singular ($s=-t-1$ gives the vanishing).
		\item if $t=-m_j-1$, then index $m_j-1$. 
		\item if $t=-m_j-2$, it is singular ($s=m_j$).
		\item if $t\leq -m_j-3$, then index $m_j$.
	\end{enumerate}

	For $\lambda_{(2)}$ we have that $$
	(\lambda_{(2)}+\delta,\alpha_1+\dots+\alpha_s)=\setlength\arraycolsep{0pt}
	\renewcommand\arraystretch{1.25}
	\left\{
	\begin{array}{llll}
	s+t \text{ if $s\leq r_j-1$,} \\
	s+t+1 \text{ if $s\geq r_j$.}
	\end{array}
	\right.
	$$
	That implies the following cases:\begin{enumerate}
		\item If $t\geq 0$, index $0$.
		\item If $-1\geq t \geq -(r_j-1)$, singular for $s=-t$.
		\item If $t=-r_j$, index $r_j-1$.
		\item If $-r_j-1\geq t\geq -m_j-1$, singular for $s=-t-1$.
		\item If $t\leq -m_j-2$, index $m_j$.
		
	\end{enumerate}
	For $r_j=m_j$ we have $Q_j(t)$, therefore the associated weight $\lambda$ is $
	\lambda=\lambda_{m_j}+t\lambda_1,
	$
	thus we have
	$$
	(\lambda+\delta,\alpha_1+\dots+\alpha_s)=\setlength\arraycolsep{0pt}
	\renewcommand\arraystretch{1.25}
	\left\{
	\begin{array}{llll}
	s+t \text{ if $s\leq m_j-1$,} \\
	s+t+1 \text{ if $s= m_j$.}
	\end{array}
	\right.
	$$
	This implies the following cases \begin{enumerate}
		\item If $t\geq 0$, we have index $0$.
		\item If $-1\geq t \geq -m_j+1$, then it is singular for $s=-t$.
		\item If $t=-m_j$, then index $m_j-1$.
		\item If $t=-m_j-1$, then it is singular for $s=m_j$.
		\item If $t\leq -m_j-2$, then index $m_j$.
	\end{enumerate}
	
	The final case is when $r_j=0$, then we have $Q_j(t+1)$ and the associated weight $\lambda$ is $\lambda_{m_j}+(t+1)\lambda_1$, therefore  	
	$$
	(\lambda+\delta,\alpha_1+\dots+\alpha_s)=\setlength\arraycolsep{0pt}
	\renewcommand\arraystretch{1.25}
	\left\{
	\begin{array}{llll}
	s+t+1 \text{ if $s\leq m_j-1$,} \\
	s+t+2 \text{ if $s= m_j$.}
	\end{array}
	\right.
	$$
	This implies the following cases \begin{enumerate}
		\item If $t\geq -1$, we have index $0$.
		\item If $-2\geq t \geq -m_j$, then it is singular for $s=-t-1$.
		\item If $t=-m_j-1$, then index $m_j-1$.
		\item If $t=-m_j-2$, then it is singular for $s=m_j$.
		\item If $t\leq -m_j-3$, then index $m_j$.
	\end{enumerate} 
\end{proof}
\begin{lema}[Bott's Theorem]
	\begin{equation}
	H^q\big(\Omega^{r_l}(t)\big)\neq 0 \text{ if } \setlength\arraycolsep{0pt}
	\renewcommand\arraystretch{1.25}
	\left\{
	\begin{array}{llll}
	q=0,\ t>r_l\\
	q=r_l,\ t=0\\
	q=n,\ t<r_l-n
	\end{array}
	\right.
	\end{equation}
\end{lema}
\begin{lema}\label{vanishing}
	Let $m_l=\dim \PP V_l$ and $k\geq3$. Suppose that $m_l\leq\sum_{i\neq l}m_i$ holds for every $l$ such that $d_l=1$. Let $r\geq2$ be an integer, $q_1,\dots,q_k$ be non negative integers such that $\sum q_l\leq r-1$, and let $r_1,\dots,r_k$ be non negative integers such that $\sum r_l=r$, then $$
	\bigotimes_{l=1,l\neq j}^kH^{q_l}(\Omega^{r_l}_{\PP^{m_l}}(2r_l-d_l(r-1)))\bigotimes H^{q_j}\big(\bigwedge^{m_j-r_j}Q_j\otimes Q_j(-d_j(r-1)+r_j-2)\big)=0, 
	$$
	for every $j\in\{1,\dots,k\}$. 
	
	Furthermore, if $k=2$ and we add the hypothesis that $(d_1,d_2)\neq (1,1)$ the result still holds.
\end{lema}

\begin{proof}
	
Suppose that the cohomology of the tensor product is non vanishing. We fix that the index $j$ will associated to the unique case coming from the cohomology table $(1)$, if not said otherwise.
	
	\textbf{Not all the cases can come from the third, fourth or fifth line of $(1)$ and from the second and third lines of $(2)$.} Suppose that one case comes from either the third, fourth or fifth lines of $(1)$, and all the remaining cases come from the second and third line $(2)$, this means that $q_l\geq r_l$, and we have that $r>q=\sum q_l\geq \sum r_l=r$.

	So at least one cohomology case must come from the other lines in $(1)$ or $(2)$.
	
	\textbf{No case can come from the first line of $(1)$.} Suppose that the only case of $(1)$ comes from the first line, this means that $-d_j(r-1)+r_j-2\geq 0$, so we obtain that $$r_j \geq (r-1)d_j+2>d_j(r-1)+1\geq (r-1)+1=r,$$ that is $r_j>r$, a contradiction.
	
	\textbf{No case can come from the fist line of $(2)$.} Suppose that we have one case coming from the first line of $(2)$ for a fixed $l$, we have that $r_l>d_l(r-1)$, then the only possiblity is that $r_l=r$ and all other $r_i=0$, for $i\neq l$ and $d_l=1$. In such case, for $i\neq l$ we have that the other cohomologies can not be on the first line, otherwise it would be $0$. Let $j$ be the only case coming from $(1)$, then for $i\neq l,j$ we have that it can not be on the second line of $(2)$, because $0=r_i=q_i=-d_i(r-1)$ and $r-1,\ d_i>0$. For $j$ we have that the second line of $(1)$ does not apply since $q_j=r_j-1=-1$ and the third line of $(1)$ implies $0=q_j=r_j$ and $-d_j(r-1)-2=-1$, then $d_j(r-1)=-1$, that is a contradiction since both terms on the left side are non negative. So in those cases we have the vanishing of the cohomology, therefore we have that one case is either on the fourth or fifth line of $(1)$ and all the remaining cases are on the third line of $(2)$. If one case is on the fifth line of $(1)$ and all the others on the third line of $(2)$, we have that $q_i=m_i$ and $q_j=m_j$ for $i\neq l$. This means $$
	m_l\geq r_l=r>\sum_{i\neq l} q_i=\sum_{i\neq l}m_i,
	$$ 
	this implies that $m_l>\sum_{i\neq l} m_i$, that is a contradiction since $d_l=1$. The case coming from $(1)$ can not be on the fourth line of $(1)$, and all the others coming from the third line of $(2)$ either, because in such case we have that $-d_j(r-1)-2=-m_j-1$, that is, $m_j-1=d_j(r-1)$, but since $r>\sum_{i\neq l} q_i=\sum_{i\neq l,j}(m_i)+m_j-1\geq m_j$, we have that $r>m_j$, and the equality can not be satisfied since $d_j\geq 1$. In the case $k=2$, notice that $r\geq m_j$ and since $d_l=1$, we must have $d_j\geq 2$. Again the wished equality $m_j-1=d_j(r-1)$ can not hold. This implies that no cohomology can come from the first line of $(2)$. 
	
	\textbf{No case can come from the second line of $(1)$.} The last remaining possibility is to have the only case of $(1)$ coming from the second line. In such case we notice that we have $q_j=r_j-1$ and no case on $(2)$ comes from the first line, thus $q_l\geq r_l$ for $l\neq j$. This, together with the fact that $\sum_{i=1}^k q_l< r$, implies that $q_l=r_l$ for $l\neq j$. We have that $-2(r_j-1)=-d_j(r-1)$, therefore $r_j=r$ and $d_j=2$, or $r_j<r$ and $d_j=1$. 
	
	In the first case we have that $r_j=r$ implies that $r_i=0$ for every $i\neq j$. This means that we have $\Omega^{r_i}(2r_i-d_i(r-1))=\OO_{\PP^{m_i}}(-d_i(r-1))$. Since $-d_i(r-1)<0$, we have that the cohomology $H^{q_i}(\OO_{\PP^{m_i}}(-d_i(r-1)))$ does not vanish just for $q_i=m_i$, but since $m_i>0$, we have that $q=\sum_{i\neq j} q_i+q_j=\sum_{i\neq j} q_i+r-1\geq r$, therefore our cases of interest have vanishing cohomology.
	
	The second possibility for this cohomology to be non vanishing is that we have $r_j<r$ and $d_j=1$. Suppose that one of these non vanishing cohomologies comes from the second line of $(2)$ for some $l$. From the conditions on $(1)$ and $(2)$ respectively, we have that $2(r_j-1)=r-1$ and $2r_l-d_l(r-1)=0$, since $d_l\geq 1$ this implies that $$
	r_j=\frac{r}{2}+\frac{1}{2},\ r_l\geq \frac{r}{2}-\frac{1}{2},
	$$ 
	therefore $r_j+r_l\geq r$. Since $k$ is at least $3$, we have another case $i$, that comes either from the second or third line of $(2)$, and we must have $r_i=0$. If it is on the second line we have that $2r_i -d_i(r-1)=0$, thus $d_i(r-1)=0$, that is a contradiction. It can not be on the third line either, since $r_i=m_i=0$ is also a contradiction. Otherwise, in case $k=2$, we assume, without loss of generality, that $j=1$ and $l=2$, then $d_1=1$ and $d_2\geq 2$, thus $r_2\geq r-1$, so we obtain $$
	r_1+r_2\geq \frac{r}{2}+\frac{1}{2}+r-1\geq r+\frac{1}{2}>r,
	$$ that is a contradiction. Therefore, no case can come from the second line on $(2)$.

	This means that all the other cases must come from the third line of $(2)$, that is $q_l=m_l=r_l$ for $l\neq j$. We notice that $r_j>\frac{r}{2}$ implies that $r_j>\sum_{l\neq j} r_l$, thus $$
	m_j\geq r_j> \sum_{l\neq j} r_l=\sum_{l\neq j}m_l,
	$$
	this is a contradiction since $d_j=1$.

\end{proof}
\begin{cor}
On the hypothesis of Lemma \ref{vanishing} we have that
$$H^q((\bigwedge^r \mathcal E^\ast)\otimes \mathcal E)=0.$$
\end{cor}
\begin{Teorema}\label{COH}
On the hypothesis of Lemma \ref{vanishing}, the induced homomorphism $$\mathcal E^\ast \otimes \mathcal E\rightarrow \mathcal I_{Z}\otimes \mathcal  E$$ induces an isomorphism at the level of global sections, where $Z$ is the zero locus of a section $s\in \mathcal E$. 
\end{Teorema}
\begin{proof}
We have the following Koszul complex $$
0=\bigwedge^{N+1}\mathcal E^\ast\xrightarrow{\varphi_{N}}\bigwedge^N\mathcal E^\ast\xrightarrow{\varphi_{N-1}}\dots\xrightarrow{\varphi_2}\bigwedge^2\mathcal  E^\ast\rightarrow\mathcal  E^\ast\rightarrow \mathcal I_Z \rightarrow 0.
$$

Let $\mathcal F_r$ to be defined as the quotient $\mathcal F_r=\bigwedge^r\mathcal E^\ast/\im \varphi_r$. Thus we obtain short exact sequences $$
\begin{matrix}
0\rightarrow\mathcal F_2\rightarrow \mathcal E^\ast\rightarrow \mathcal I_Z\rightarrow 0\\
0\rightarrow\mathcal F_{r+1}\rightarrow\bigwedge^r \mathcal E^\ast\rightarrow \mathcal F_{r}\rightarrow 0,
\end{matrix}
$$
for $r=2,\dots, N$.

Tensoring the second short exact sequence by $\mathcal E$,  we obtain the long exact sequence of cohomologies \begin{equation*}
\begin{split}
\dots\rightarrow H^{r-2}(\bigwedge^r \mathcal E^\ast\otimes \mathcal E)\rightarrow H^{r-2}(\mathcal F_r\otimes \mathcal E)\rightarrow H^{r-1}(\mathcal F_{r+1}\otimes \mathcal E)\rightarrow \\ \rightarrow H^{r-1}(\bigwedge^r\mathcal E^\ast\otimes \mathcal E)\rightarrow H^{r-1}(\mathcal F_r\otimes \mathcal E)\rightarrow H^r(\mathcal F_{r+1}\otimes\mathcal E)\rightarrow\dots
\end{split}
\end{equation*}
By the previous lemma we have that both terms on the left are zero, therefore we have that $$
H^{r-2}(\mathcal F_r\otimes \mathcal E)\cong H^{r-1}(\mathcal F_{r+1}\otimes \mathcal E), \ H^{r-1}(\mathcal F_r\otimes \mathcal E)\subset H^r(\mathcal F_{r+1}\otimes\mathcal E). 
$$
This implies that $$
\begin{matrix}
H^0(\mathcal F_2\otimes \mathcal E)\cong\dots\cong H^{N-1}(\mathcal F_{N+1}\otimes\mathcal  E)=0\\
H^1(\mathcal F_2\otimes\mathcal  E)\subset\dots\subset H^N(\mathcal F_{N+1}\otimes \mathcal E)=0.
\end{matrix}
$$
If we consider now the long exact sequence of cohomologies from $0\rightarrow\mathcal F_2\otimes \mathcal E\rightarrow \mathcal E^\ast\otimes \mathcal E\rightarrow \mathcal I_Z\otimes \mathcal E\rightarrow 0$, we obtain $$
H^0(\mathcal F_2\otimes \mathcal  E)\rightarrow H^0(\mathcal E^\ast\otimes \mathcal E)\rightarrow H^0(\mathcal I_Z\otimes \mathcal E)\rightarrow H^1(\mathcal F_2\otimes \mathcal E), 
$$
since the end terms are zero, we obtain the desired isomorphism.
\end{proof}

To make a comparison with the symmetric case, our next objective is to prove the extension of Corollary \ref{cor sym} to the multisymmetric case. That is, we will show if two sections $s,\ t$, that arise from the respective tensors $S,\ T\in \Sym^{d_1}V_1\otimes \dots\otimes \Sym^{d_k}V_k$, have the same image under $\psi$, that is, we have the equality of the zero locus $Z(s)=Z(t)$, then $s=\lambda t$.

\begin{lema}\label{end0}
	Let $\mathcal E_i=\pi^\ast_iQ_i(d_1,\dots,d_i-1,\dots,d_k)$. If $\dim \PP V_j\geq 2$ for all $j$, then $$
	H^0(\Hom(\mathcal E, \mathcal E_j))=H^0(\Hom(\mathcal E_j,\mathcal E_j))=\CC.
	$$
	Moreover, if we assume that $i\neq j$, then $$
	H^0(\Hom(\mathcal E_i,\mathcal E_j))=0.
	$$
\end{lema}
\begin{proof}
 	For the second equality we have that
	\begin{equation*}
	\begin{aligned}
	\Hom(\mathcal E_i,\mathcal E_j)=\pi^\ast_i Q^\vee_i\otimes \pi^\ast_j Q_j(0,\dots,0,1,0,\dots,0,-1,0,\dots,0)=\\\OO_{\PP V_1}\otimes\dots\otimes\pi^\ast_iQ_i^\vee(1)\otimes\dots\otimes\pi^\ast_jQ_j(-1)\otimes\dots\otimes \OO_{\PP V_k}.
	\end{aligned}
	\end{equation*}
	We notice that, for all $i$, we have that $H^0(Q_i^\vee(1))=H^0(\Omega^1(2))\neq0$. Meanwhile, if $\dim \PP V_j\geq2$, then $H^0(Q_j(-1))=0$, thus by the K\"unneth's formula we have that $$
	H^0(\Hom(\mathcal E_i,\mathcal E_j))=0.
	$$
	
	On the other hand, $$\Hom(\mathcal E_j,\mathcal E_j)=\OO_{\PP V_1}\otimes\dots \otimes \big(\pi_j^\ast (Q_j^\ast\otimes Q_j)\big)\otimes\dots\otimes \OO_{\PP V_k}=\Hom(Q_j,Q_j),$$
	since the bundle $Q_j$ is simple we obtain the desired result.

\end{proof} 

\begin{lema}\label{end}
	Let $\rho\in End(H^0(\mathcal E))$ be a endomorphism of $H^0(\mathcal E)$, suppose that $f,\ g\in \Sym^{d_1}V_1\otimes\dots\otimes\Sym^{d_k}V_k$ are tensors such that $\rho(s_f)=\rho(s_g)$, then $s_f=\lambda s_g$ for $\lambda \in \CC$.
\end{lema}
\begin{proof}

		Let $I_1$ be the set of indices such that $\dim \PP V_i=1$ and $I_2$ be the set of indices such that $\dim \PP V_i\geq 2$. By the previous lemma, we have that $H^0(Hom(\mathcal E_i, \mathcal E_j))=0$, whenever $j\in I_2$, that is, no map can act there besides its own endomorphism.
	
	Now we consider a section $s_f$ coming from a tensor $f\in \Sym^{d_1}V_1\otimes\dots\otimes \Sym^{d_k}V_k$. We recall that the map $\varphi$ associates $f$ to the diagonal map of its flattenings in each coordinate $l$, that is, $f:\Sym^{d_1}V_1\otimes\Sym^{d_l-1}V_l\otimes\dots\otimes \Sym^{d_k}V_k\rightarrow V_l$. This means that $\varphi(f)=s_f$ can be interpreted as the diagonal element $s_f=(f,\dots,f)$, where $f$ in the $l$ entry of this vector means the section of $\mathcal E_l$ corresponding to $f$.
	
	Suppose that the first $l$ indices are in $I_1$ and the others in $I_2$. Applying $\rho$ to $\varphi(f)$ we obtain that $$
	\rho(\varphi(f))=(M_1(f),\dots,M_l(f),\lambda_{l+1}f,\dots,\lambda_{k}f)=(g,\dots,g)=s_g,
	$$
	where $g\in \Sym^{d_1}V_1\otimes\dots\otimes \Sym^{d_k}V_k$ is a tensor, thus from the previous lemma we have that all the maps $\lambda_i$ must be multiplication by scalars, this means that $\lambda f=g$, for some $\lambda\in \CC$.
	
	It remains the case when $I_2=\emptyset$. In such case we notice that $$Hom(\mathcal E_i,\mathcal E_j)=(0,\dots,0,Q^\ast_i(1),0,\dots,0,Q_j(-1),0,\dots,0);$$ since the dimension of each $\PP V_i$ is $1$, we have $Q_j(-1)=\OO_{\PP^1}$, moreover $Q^\ast_i(1)=\Omega^1_{\PP^1}(2)=\OO_{\PP^1}$. We recall that both of those bundles are $1$-dimensional at the level of global sections, that is, $\dim H^0(\OO_{\PP^1})=1$, therefore $\dim H^0(Hom(\mathcal E_i,\mathcal E_j))=1$. This implies that if $\rho(s_g)=s_f$, then $
	s_f=\lambda s_g
	$
	is the only possible image.
	
\end{proof}

Combining the previous results together, we obtain the next theorem.
\begin{Teorema}\label{psi}
	Let $S,T\in\Sym^{d_1}V_1\otimes\dots\otimes \Sym^{d_k}V_k$ be two general tensors. Assume that $m_l\leq\sum_{i\neq l}m_i$ holds for every $l$ such that $d_l=1$, $k\geq 3$, and that $m_j\geq 1$ for all $j$. Let $s,t\in \mathcal E$ be the sections coming from the tensors, $S$ and $T$, and assume that $Z(s)=Z(t)$, then $s=\lambda t$, for $\lambda\in \CC^\ast$.
	
	Additionally, if $k=2$ and we also consider the hypothesis that $(d_1,d_2)\neq (1,1)$, then the result still holds.
\end{Teorema}
\begin{proof}
	
	The Theorem \ref{COH} says that the map $H^0(\End(\mathcal E))\rightarrow H^0(\mathcal I_Z\otimes E)$ defined by $\rho\mapsto \rho(s_g)$ is an isomorphism. This means that if $s,\ t$ are two tensors such that $Z(s)=Z(t)$, then there exists a morphism $\rho\in \End \mathcal E$ such that $$\rho(t)=s.$$ Furthermore, from the Lemma \ref{end} we obtain that $s=\lambda t$.
	
\end{proof}

With all those results in mind, we can conclude with a note that combining together Theorem \ref{segrever} with the Theorem \ref{psi} we obtain the Theorem \ref{fmulti}. This finishes the proof of the main results.

\subsection{A remark on sections coming from tensors}

We make a brief remark that the Lemma \ref{end} does not mean that all the morphisms in $\End(\mathcal E)$ are multiplication by scalars, since the map $\varphi$ is not surjective in general. Indeed, for each space $V_l$, we can compute what is the image of $\varphi_l:\Sym^{d_l}V_l\rightarrow H^0(Q_l(d_l-1))$. 

We have the Euler exact sequence $$
0\rightarrow H^0(\OO(d_l-2))\rightarrow H^0(\OO(d_l-1)\otimes V_l)\rightarrow H^0(Q(d_l-1))\rightarrow 0.
$$ 
Moreover we have an isomorphism $$H^0(\OO(d_l-1)\otimes V_l)\cong\Sym^{d_l-1}V_l^\ast \otimes V_l\text{ and } H^0(\OO(d_l-2))\cong \Sym^{d_l-2}V_l$$

In terms of Young diagrams, $\Sym^{d_l-2}V_l$ has the representation $$
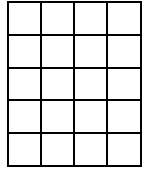
\yng(4,4,4,4,4)
$$
where the tableaux has $d_l-2$ columns and $\dim V_l-1$ rows. Moreover, $\Sym^{d_l-1}V_l^\ast \otimes V_l$ is represented by $$
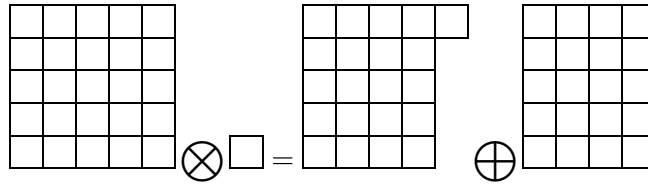
\yng(5,5,5,5,5) \bigotimes \yng(1)=\yng(5,4,4,4,4)\bigoplus \yng(4,4,4,4,4)$$
by Pieri's formula. Thus $H^0(Q_l(d_l-1))$ is given by $$
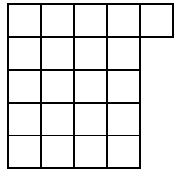
\yng(5,4,4,4,4)
$$

We can compute what exactly is $H^0(Q(d_l-1))$ in terms of irreducible $SO(m_l)$-representations by restricting the product $\Sym^{d_l-1}V_l\otimes V$ as $SL(m_l)$-representations to $SO(m_l)$-representations. Indeed, we have that as $SL(m_l)$-representations $$
\Sym^{d_l-1}V_l\otimes V_l=\Sym^{d_l}V_l\oplus \Gamma^{(d_l-2,1,0,\dots,0)},
$$
the first summand restricts to $$Res^{SL(m_l)}_{SO(m_l)}(\Sym^{d_l}V_l)=H_{d_l}\oplus H_{d_l-2}\oplus\dots\oplus \setlength\arraycolsep{0pt}
\renewcommand\arraystretch{1.25}
\left\{
\begin{array}{llll}
H_1 \text{ if $d_l$ is odd} \\
H_0 \text{ if $d_l$ is even}
\end{array}
\right.
$$

To compute the restriction of $\Gamma^{(d_l-2,1,0,\dots,0)}$ we now utilise the fact that $$Res^{SL(m_l)}_{SO(m_l)}(\Gamma^\lambda)=\bigoplus_{\overline\lambda} N_{\lambda\overline\lambda}\Gamma^{\overline\lambda},$$ with $N_{\lambda\overline\lambda}=\sum_\delta N_{\delta,\overline\lambda\lambda},$ where $N_{\delta\overline\lambda\lambda}$ is the Littlewood-Richardson coefficient, $\overline\lambda=(\overline\lambda_1\geq\overline\lambda_2\geq\dots\geq\overline\lambda_t\geq0)$, $m_l=2t$ or $2t+1$, and $\delta=(\delta_1\geq\delta_2\geq\dots\geq0)$, with $\delta_i$ even for all $i$. For more details about the restriction and Littlewood-Richardson coefficients we suggest \cite{Fulton}.

In our setting $\lambda=(d_l-2,1,\dots,0)$ is represented as the Young tableaux given by $$\yng(10,1)$$
with $d_l-1$ boxes in the first line and 1 box on the second. 

If $\delta=(0)$, the only possible tableaux for $\overline{\lambda}$ is the tableaux of $\lambda$ itself, that is, $\overline\lambda=(d_l-2,1,0,\dots,0)$.

If $\delta=2h$, for $h\geq1$ and $d_l-2h-2\geq 0$, then there are two other possibilities for $\overline{\lambda}$, indeed $\overline{\lambda}=(d_l-2h,0,\dots,0)$ or $\overline{\lambda}=(d_l-2h-2,1,0,\dots,0)$, indeed this results in
$$\begin{Young}
&&&&&&$1$&$1$&$1$&$1$\cr
1\cr
\end{Young}
$$
and $$\begin{Young}
&&&&&&$1$&$1$&$1$&$1$\cr
2\cr
\end{Young}
$$
Notice that, if we tried to add more columns to the second line, we would have a box of index $1$ and a box of index $2$ in the first row, that is forbidden by the Littlewood-Richardson rule since they come from the same column. Similarly, if we tried to add more rows, we would have the same problem. Therefore those are all the possible $\overline\lambda$. Let $$\Gamma=\Gamma^{(d_l-2,1,0,\dots,0)}\oplus\Gamma^{(d_l-4,1,0,\dots,0)}\oplus\dots\oplus\setlength\arraycolsep{0pt}
\renewcommand\arraystretch{1.25}
\left\{
\begin{array}{llll}
\Gamma^{(1,1,0,\dots,0)} \text{ if $d_l$ is odd} \\
\Gamma^{(0,1,0,\dots,0)} \text{ if $d_l$ is even}
\end{array}
\right.
$$ and
$$
\mathcal H_{d_l-2}=H_{d_l-2}\oplus H_{d_l-4}\oplus\dots\oplus \setlength\arraycolsep{0pt}
\renewcommand\arraystretch{1.25}
\left\{
\begin{array}{llll}
H_1 \text{ if $d_l$ is odd} \\
H_2 \text{ if $d_l$ is even}.
\end{array}
\right.
$$
We obtain that $$
Res^{SL(m_l)}_{SO(m_l)}(\Gamma^{(d_l-2,1,0,\dots,0)})=\mathcal H_{d_l-2}\oplus \Gamma,
$$
moreover $$
Res^{SL(m_l)}_{SO(m_l)}(\Sym^{d_l-1}V_l\otimes V_l)=\mathcal H_{d_l-2}\oplus \Gamma\oplus H_{d_l}\oplus H_{d_l-2}\oplus\dots\oplus \setlength\arraycolsep{0pt}
\renewcommand\arraystretch{1.25}
\left\{
\begin{array}{llll}
H_1 \text{ if $d_l$ is odd} \\
H_0 \text{ if $d_l$ is even}
\end{array}
\right..
$$
Now we can compute the $H^0(Q(d_l-1))$, from the Euler exact sequence we have that it is given by the difference $\Sym^{d_l-1}V_l\otimes V_l-\Sym^{d_l-2}V_l$, that is,\begin{equation*}
\begin{aligned}
H^0(Q(d_l-1))=&\mathcal H_{d_l-2}\oplus \Gamma\oplus  H_{d_l}\oplus H_{d_l-2}\oplus  \dots \oplus \setlength\arraycolsep{0pt}
\renewcommand\arraystretch{1.25}
\left\{
\begin{array}{llll}
H_1 \text{ if $d_l$ is odd} \\
H_0 \text{ if $d_l$ is even}
\end{array}
\right. \\-&
H_{d_l-2}\oplus H_{d_l-4} \oplus\dots\oplus \setlength\arraycolsep{0pt}
\renewcommand\arraystretch{1.25}
\left\{
\begin{array}{llll}
H_1 \text{ if $d_l$ is odd} \\
H_0 \text{ if $d_l$ is even},
\end{array}
\right.\\
\end{aligned}
\end{equation*}
therefore we obtain that, $$H^0(Q(d_l-1))=\mathcal H_{d_l}\oplus \Gamma,$$ where $\mathcal H_{d_l}$ is the subrepresentation of the sections coming from tensors by the Lemma \ref{symmetric}.

\printbibliography

@article{Qi,
title = "Eigenvalues and invariants of tensors",
journal = "Journal of Mathematical Analysis and Applications",
volume = "325",
number = "2",
pages = "1363 - 1377",
year = "2007",
author = "Liqun Qi",
%issn = "0022-247X",
%doi = "https://doi.org/10.1016/j.jmaa.2006.02.071",
%url = "http://www.sciencedirect.com/science/article/pii/S0022247X06001764",

%keywords = "Eigenvalue, Tensor, Invariant, Supermatrix, Rank",
%abstract = "A tensor is represented by a supermatrix under a co-ordinate system. In this paper, we define E-eigenvalues and E-eigenvectors for tensors and supermatrices. By the resultant theory, we define the E-characteristic polynomial of a tensor. An E-eigenvalue of a tensor is a root of the E-characteristic polynomial. In the regular case, a complex number is an E-eigenvalue if and only if it is a root of the E-characteristic polynomial. We convert the E-characteristic polynomial of a tensor to a monic polynomial and show that the coefficients of that monic polynomial are invariants of that tensor, i.e., they are invariant under co-ordinate system changes. We call them principal invariants of that tensor. The maximum number of principal invariants of mth order n-dimensional tensors is a function of m and n. We denote it by d(m,n) and show that d(1,n)=1, d(2,n)=n, d(m,2)=m for m⩾3 and d(m,n)⩽mn−1+⋯+m for m,n⩾3. We also define the rank of a tensor. All real eigenvectors associated with nonzero E-eigenvalues are in a subspace with dimension equal to its rank."
}

@INPROCEEDINGS{Lim,  author={Lek-Heng Lim},  booktitle={1st IEEE International Workshop on Computational Advances in Multi-Sensor Adaptive Processing, 2005.},   title={Singular values and eigenvalues of tensors: a variational approach},   year={2005},  volume={},  number={},  pages={129-132}}

@article{Paoletti,
author = {Ottaviani, Giorgio and Paoletti, Raffaella},
year = {2015},
month = {03},
pages = {},
title = {A Geometric Perspective on the Singular Value Decomposition},
volume = {47},
journal = {Rendiconti dell'Istituto di Matematica dell'Universita di Trieste},
%doi = {10.13137/0049-4704/11222}
}

@misc{Beorchia,
      title={Eigenschemes of ternary tensors}, 
      author={Valentina Beorchia and Francesco Galuppi and Lorenzo Venturello},
      year={2020},
      eprint={2007.12789},
      archivePrefix={arXiv},
     % primaryClass={math.AG}
}

@article{Draisma,
author = {Draisma, Jan and Ottaviani, Giorgio and Tocino, Alicia},
year = {2017},
month = {11},
pages = {},
title = {Best rank-$k$ approximations for tensors: generalizing Eckart-Young},
volume = {5},
journal = {Research in the Mathematical Sciences},
%doi = {10.1007/s40687-018-0145-1}
}

@book{Gelfand,
author={Gelfand, Israel M. and Kapranov, Mikhail M. and Zelevinsky, Andrei V.},
year={1994},
title = {Discriminants, Resultants, and Multidimensional Determinants},
publisher={Birkh\" auser, Boston, MA},
}

@book{Fulton,
  title={Representation Theory: A First Course},
  author={Fulton, William and Harris, Joe},
  series={Graduate Texts in Mathematics},
  year={1991},
  publisher={Springer New York}
}

@article{Oeding,
title = {Eigenvectors of tensors and algorithms for Waring decomposition},
journal = {Journal of Symbolic Computation},
volume = {54},
pages = {9-35},
year = {2013},
author = {Luke Oeding and Giorgio Ottaviani},
}

@article{Friedland,
author = {Friedland, Shmuel and Ottaviani, Giorgio},
year = {2014},
pages = {1209-1242},
title = {The Number of Singular Vector Tuples and Uniqueness of Best Rank-One Approximation of Tensors},
journal = {Foundations of Computational Mathematics},
volume = {14},
}

@article{Horobet,
author = {Draisma, Jan and Horobet, Emil and Ottaviani, Giorgio and Sturmfels, Bernd and Thomas, Rekha},
year = {2013},
month = {08},
pages = {},
title = {The Euclidean Distance Degree of an Algebraic Variety},
volume = {16},
journal = {Foundations of Computational Mathematics},
}

@article{Boralevi,
author = {Boralevi, Ada and Draisma, Jan and Horobet, Emil and Robeva, Elina},
year = {2015},
month = {12},
pages = {},
title = {Orthogonal and unitary tensor decomposition from an algebraic perspective},
volume = {222},
journal = {Israel Journal of Mathematics},
}

@article{Vannieuwenhoven,
author = {Vannieuwenhoven, Nick and Nicaise, Johannes and Vandebril, Raf and Meerbergen, Karl},
year = {2014},
month = {07},
pages = {886-903},
title = {On Generic Nonexistence of the Schmidt--Eckart--Young Decomposition for Complex Tensors},
volume = {35},
journal = {SIAM Journal on Matrix Analysis and Applications},
}

@article{Abo,
author = {Abo, Hirotachi and Seigal, Anna and Sturmfels, Bernd},
year = {2017},
month = {01},
pages = {1-25},
title = {Eigenconfigurations of tensors},
journal = {Contemporary Mathematics},
}

@book{Weyman, place={Cambridge}, series={Cambridge Tracts in Mathematics}, title={Cohomology of Vector Bundles and Syzygies}, publisher={Cambridge University Press}, author={Weyman, Jerzy}, year={2003}, collection={Cambridge Tracts in Mathematics}}

@book{landsbergtensors,
  title={Tensors: Geometry and Applications},
  author={Landsberg, Joseph M.},
  series={Graduate studies in mathematics},
year={2012},
volume={128},
}

@misc{Sodomaco,
      title={Asymptotics of degrees and ED degrees of Segre products}, 
      author={Giorgio Ottaviani and Luca Sodomaco and Emuanuele Ventura},
      year={2020. \text{To appear in Advances in Applied Mathematics}},
archivePrefix={arXiv},
      eprint={2008.11670}, 
notes={To appear in Advances in Applied Mathematics},
}

@misc{sodomaco2019product,
      title={On the product of the singular values of a binary tensor}, 
      author={Luca Sodomaco},
      year={2019},
archivePrefix={arXiv},
      eprint={1906.05181},
}

@article{ABO2016121,
title = {Eigenschemes and the Jordan canonical form},
journal = {Linear Algebra and its Applications},
volume = {496},
pages = {121-151},
year = {2016},
author = {Hirotachi Abo and David Eklund and Thomas Kahle and Chris Peterson},
}

@article{CARTWRIGHT2013942,
title = {The number of eigenvalues of a tensor},
journal = {Linear Algebra and its Applications},
volume = {438},
number = {2},
pages = {942-952},
year = {2013},
note = {Tensors and Multilinear Algebra},
author = {Dustin Cartwright and Bernd Sturmfels},
}

@article{Chang,
author = {Chang, Kungching and Qi, Liqun and Zhang, Tan},
title = {A survey on the spectral theory of nonnegative tensors},
journal = {Numerical Linear Algebra with Applications},
volume = {20},
number = {6},
pages = {891-912},
year = {2013}
}

@book{qi2017tensor,
  title={Tensor Analysis: Spectral Theory and Special Tensors},
  author={Qi, Liqun and Luo, Ziyan},
  series={Other Titles in Applied Mathematics},
  year={2017},
  publisher={Society for Industrial and Applied Mathematics}
}

 \end{document}